\documentclass[a4paper,10pt]{article}

\usepackage[utf8]{inputenc}
\usepackage{amssymb} 
\usepackage{amsmath}
\usepackage{amsthm}
\usepackage{mathrsfs}
\usepackage{graphicx}
\usepackage{subfigure}
\usepackage{bm}
\usepackage{CJKutf8}

\theoremstyle{plain} 
\newtheorem{thm}{Theorem}
\newtheorem{proposition}[thm]{Proposition}
\newtheorem{corollary}[thm]{Corollary}

\theoremstyle{definition}
\newtheorem{remark}[thm]{Remark}
\newtheorem{example}[thm]{Example}
\newtheorem{definition}[thm]{Definition}
\newtheorem{notation}[thm]{Notation}

\def\mbb #1{\mathbb{#1}}
\def\msf #1{\mathsf{#1}}

\def\mscr #1{\mathscr{#1}}

\def\C{\mbb{C}}
\def\CP{\mbb{CP}}
\def\R{\mbb{R}}
\def\Z{\mbb{Z}}
\def\N{\mbb{N}}

\newcommand{\Spec}{\mathrm{Spec\,}}
\newcommand{\sminus}{\smallsetminus}

\newcommand*\hg{\begingroup
        \dopFq}
\def\dopFq#1#2#3{\Big(\genfrac..{0pt}{}{#1}{#2}\Big |\, #3\Big)%
        \endgroup}

\newcommand{\bmbeta}[1]{ \hat{\bm\beta}{\raisebox{2pt}{${}^{#1}$}} }

\begin{document}

\author{Martin Klime\v s}

\title{Confluence of singularities in hypergeometric systems}

\maketitle

\vskip1pt
\begin{abstract}
A system in a Birkhoff normal form with an irregular singularity of Poincaré rank 1 at the origin and a regular singularity at infinity is 
through the Borel-Laplace transform dual to a system in an Okubo form. 
Sch\"afke has showed that the Birkhoff system can also be obtained from the Okubo
system by a simple limiting procedure. 
The Okubo system comes naturally with two kinds of mixed solution bases,
both of which converge under the limit procedure to the canonical solutions of the limit Birkhoff system on sectors near the irregular singularity at the origin.
One can then define Stokes matrices of the Okubo system as connection matrices between different branches of the mixed solution bases
and use them to relate the monodromy matrices of the Okubo system to the usual Stokes matrices of the limit system at the irregular singularity.
This is illustrated on the example of confluence in the generalized hypergeometric equation.

\end{abstract}

\renewcommand{\thefootnote}{\fnsymbol{footnote}} 
\footnotetext{\emph{Key words\/}: Linear differential equations, generalized hypergeometric equation, confluence, Stokes matrices, monodromy.}     
\renewcommand{\thefootnote}{\arabic{footnote}}

\section{Introduction}

A linear differential system
\begin{equation}\label{eq:hg-1}
(s-B)\frac{dv}{ds}=(A+\rho)v, \qquad (s,v)\in\CP^1\times\C^n, 
\end{equation}
where $A,B$ are constant $n\times n$-matrices, $B$ is diagonal, $\rho\in\C$ a parameter,
is called an \emph{Okubo system}, or also a \emph{hypergeometric system}. 
Such systems appear as a natural generalization of the hypergeometric equation. 
It is known \cite{Ko2}, that every single Fuchsian differential equation can be reduced to such a system. 

The assumption that $B$ is diagonal (or semisimple) assures that the 1-form $(s-B)^{-1}\,ds$ has only simple poles (placed at the eigenvalues of $B$ and at $\infty$), 
i.e. that all the singularities of the Okubo system \eqref{eq:hg-1} are Fuchsian.

The Okubo system \eqref{eq:hg-1} appears also as a dual to a system in \emph{Birkhoff normal form}
\begin{equation}\label{eq:hg-2}
z^2\frac{d\psi}{dz}=(B+zA)\psi, \qquad (z,\psi)\in\CP^1\times\C^n, 
\end{equation} 
which has an irregular singular point at 0 and a Fuchsian singular point at $\infty$, 
through the Laplace transform
\begin{equation*}
\psi(z)=z^{-1-\rho}\int_0^\infty v(s,\rho)\,e^{-\frac{s}{z}}\,ds, \qquad |\arg s-\arg z|<\tfrac{\pi}{2}.
\end{equation*}
This fact can be used to express the Stokes and connection matrices of the Birkhoff system in terms of connection matrices and monodromies of the dual Okubo system \cite{BJL, Sch1}.

Sch\"afke \cite{Sch2} has remarked that the the system \eqref{eq:hg-2} can be also obtained from \eqref{eq:hg-1} by the following
confluence procedure:
$$\text{let}\quad s=\rho z, \quad\text{and}\quad y(z,\rho)=s^{-\rho}v(s,\rho),$$
then $y$ satisfies
\begin{equation}\label{eq:hg-3}
z(z-\frac{1}{\rho}B)\frac{dy}{dz}=(B+zA)y, 
\end{equation}  
which becomes \eqref{eq:hg-2} at the limit when $\rho\to\infty$.
 
In case of rank $n=2$ and $B$ with two distinct eigenvalues, 
this confluence procedure corresponds exactly to the confluence of the (Gauss') hypergeometric equation to the (Kummer's) confluent hypergeometric equation.

Aside from the usual local Levelt bases at each of the singularities, the Okubo system \eqref{eq:hg-1} has two other kinds of natural solution bases,
so called \emph{mixed bases} \cite{BJL, Sch1, OTY}:
The first one, called \emph{Floquet basis}, consists of the Floquet solutions (singular Levelt solutions) at different finite singularities $\lambda_j$ (eigenvalues of $B$). The other,  called \emph{co-Floquet basis}, is in a sense dual; a co-Floquet solution at a singularity $\lambda_j$ is one that is 
analytic at al other singularities $\lambda_k$, $k\neq j$.
Sch\"afke \cite{Sch2} has studied the limits of these mixed bases in the confluent family \eqref{eq:hg-3}
in the case where all the eigenvalues of  $B$ are distinct 
and has shown that they both tend to the canonical solution basis of the limit system (Borel sum of a formal fundamental solution) 
on sectors at the irregular singularity $z=0$:
the Floquet basis when $\rho\to+\infty$, and the co-Floquet basis when $\rho\to-\infty$.

This article exposes these results while extending them to a more general situation, where $B$ is allowed to have multiple eigenvalues, and $\rho$ can go to infinity along any fixed direction in one of two sectors of opening $>\pi$ covering a neighborhood of 
$\infty$ on the Riemann sphere $\CP^1$.
In an analogy with \cite{LR2, HLR} it is natural to introduce parametric \emph{Stokes matrices} of the confluent family \eqref{eq:hg-3},
as connection matrices between different branches of the Floquet (resp. co-Floquet) basis far from the origin.
These parametric Stokes matrices are closely related to the monodromy 
 of the family \eqref{eq:hg-3}:
in general, the monodromy matrices of the Floquet and co-Floquet bases
can be expressed as products of these Stokes matrices and formal monodromy matrices.
While the monodromy matrices diverge when $\rho\to\infty$
(because of their formal monodromy parts which are exponential functions of $\rho$)
these parametric Stokes matrices tend to  the usual Stokes matrices of the limit system, and can be easily obtained from them
(Proposition \ref{prop:hg-stokesmatrices} ).

These results are illustrated in Section 2 on explicit calculations in the case of the generalized hypergeometric equation, %\eqref{eq:hg-hge}, 
previously studied by  Duval \cite{Du}.
Duval considered the problem of convergence of the monodromy matrices to the Stokes matrices without separating formal monodromy part and the Stokes part. Therefore she could only consider limits when $\rho\to\pm\infty$ following a discrete set of values on which the formal monodromy part is constant.

\begin{remark}
A different confluence procedure of the type 
\begin{equation}\label{eq:hg-4}
(z^2-\epsilon)\frac{dy}{dz}=\Omega(z)y, \qquad\C\ni\epsilon\to 0,
\end{equation}  
was investigated in e.g. \cite{Pa, Gl1, Gl2, LR2, HLR, Kl1, Kl2}.
In case of $B$ having only two eigenvalues (one of which can be always shifted to 0), the confluence procedure \eqref{eq:hg-3}
can be considered as a special case of \eqref{eq:hg-4}.
In this case our perspective essentially coincides with that of \cite{LR2, HLR}.
In particular,
this includes the confluence in the Gauss' hypergeometric equations \cite{MR,Ra,Zh,LR1} and in the generalized hypergeometric equation \cite{Du}.
\end{remark}

\section{General theory}

Let the matrix $B$ be diagonal with eigenvalues $\lambda_1,\ldots,\lambda_p$ of respective multiplicities $n_1,\dots,n_p$, and let the matrix $A$ be partitioned into blocks accordingly
\begin{equation*}
B=\begin{pmatrix}
\lambda_1I_{n_1} && \\ & \ddots & \\ && \lambda_pI_{n_p}
\end{pmatrix},\qquad
A=\begin{pmatrix}
A_{11} &\ldots& A_{1p} \\ \vdots &  &\vdots \\ A_{p1} &\ldots & A_{pp} 
\end{pmatrix}.
\end{equation*}
The following \textbf{assumption} is made throughout the text:
\begin{equation}\label{eq:hg-asum}
\textit{No two eigenvalues of $A_{jj}$ differ by a non-zero integer,  $1\leq j\leq p$.}
\end{equation}

\begin{notation}
For any  $n\!\times\! n$-matrix $X$, let $(X_{ij})_{1\leq i,j\leq p}$ be its bloc-partition according to $B$, and
denote \,$X_{\cdot j}=$ the $j$-th bloc column of $X$.
\end{notation}

%When $X$ is a fundamental solution matrix of some system, $X_{\cdot j}$ will often be called simply a "solution", while it is in fact bloc of lineary independent solutions  

\goodbreak

\subsection{Fundamental solution of the limit system \eqref{eq:hg-2}}\label{sec: hg-2.1}

It is well-known (see for example \cite{Ba})
that the system \eqref{eq:hg-2} can be bloc-diagonalized by means of a formal power series transformation
$\psi=\hat T(z)\phi$, with 
$$\hat T(z)=\sum_{k=0}^{+\infty}T^{(k)}z^k, \qquad T^{(0)}=0.$$
Under the assumption \eqref{eq:hg-asum}, the formally transformed system can be given the following Birkhoff form
\begin{equation}\label{eq:hg-2formal}
z^2\frac{d\phi}{dz}=(B+zA_D)\phi, \qquad \text{with}\quad 
A_D=\begin{pmatrix}
A_{11} && \\  & \ddots & \\  && A_{pp} 
\end{pmatrix}.
\end{equation} 
Therefore the system \eqref{eq:hg-2} has a formal fundamental solution $\hat\Psi(z)$ whose $j$-th bloc-column is given by
\begin{equation*}
\hat\Psi_{\cdot j}(z)=\hat T_{\cdot j}(z) z^{A_{jj}} e^{-\frac{\lambda_j}{z}}.
\end{equation*}

While $\hat T(z)$ is in general divergent, it is Borel summable. 
More precisely each its column $\hat T_{\cdot j}(z)$ is Borel summable 
in all directions $\alpha$ with $e^{i\alpha}\R^+$ disjoint from all $\lambda_i-\lambda_j$, $i\neq j$ (such direction $\alpha$ will be called \emph{non-singular}).
Let 
\begin{equation}\label{eq:hg-U}
U_{\cdot j}(s)=\sum_{k=0}^{+\infty} \frac{T^{(k)}_{\cdot j}}{k!}s^k
\end{equation}
be the formal Borel transform of $z\,\hat T_{\cdot j}(z)$, convergent near $s=0$ and extended analytically on the universal covering of 
$\C\sminus\{\lambda_i-\lambda_j \mid 1\leq i\leq p,\ i\neq j\}$.
The matrix function $U(s)$ is a solution to linear system
\begin{equation*}
s\frac{dU}{ds}-B\frac{dU}{ds}+\frac{dU}{ds}B=AU-U\!A_D,
\end{equation*} 
with Fuchsian singularities at the points $\lambda_i-\lambda_j$ and $\infty$.
In particular, $U$ has only a moderate growth at each of the singularities.
Therefore the Borel sum of  $\hat T_{\cdot j}(z)$ in a non-singular direction $\alpha$  is  well-defined by the Laplace integral
\begin{equation*}
T_{[\alpha],\cdot j}(z)=\frac{1}{z}\int_0^{+\infty e^{i\alpha}}\!\!\! U_{\cdot j}(s)\,e^{-\frac{s}{z}} ds,
%=:\frac{1}{z}\cal L_\alpha[U_{\cdot j}](z),
\end{equation*}
which converges and is bounded for $z$ in the open half-plane bisected by $e^{i\alpha}\R^+$, 
and whose value is independent of when the direction $\alpha$ varies a bit.
In another words, the sectoral transformation $T_{[\alpha]}$ depends only on the homotopy class $[\alpha]$ of the direction $\alpha\in
\R\sminus\{\text{singular directions}\}$,
and one can consider it as defined on a sector in the $z$-plane
\begin{equation}\label{eq:hg-sector}
\mscr S_{[\alpha]}(\infty):=
\bigcup_{\alpha'\in[\alpha]}\{\Re(e^{-i\alpha'}z)>0\},\end{equation}
of opening $>\pi$.
Once a branch of $\log z$ is fixed, the system \eqref{eq:hg-2} has on each of these sectors a \emph{canonical solution basis} $\Psi_{[\alpha]}(z)$
\begin{equation}\label{eq:hg-Psi_alpha}
\Psi_{[\alpha],\cdot j}(z)=T_{[\alpha],\cdot j}(z)\, z^{A_{jj}} e^{-\frac{\lambda_j}{z}}=
\frac{1}{z}\int_0^{+\infty e^{i\alpha}}\!\!\! U_{\cdot j}(s)\,e^{-\frac{s}{z}} ds\cdot z^{A_{jj}}e^{-\frac{\lambda_j}{z}}.
\end{equation}

For every pair of non-singular directions $\alpha_1,\, \alpha_2$ there is \emph{Stokes matrix} $S_{[\alpha_1][\alpha_2]}(\infty)$
\begin{equation}\label{eq:hg-stokesmatrices}
\Psi_{[\alpha_2]}=\Psi_{[\alpha_1]}\cdot S_{[\alpha_1][\alpha_2]}(\infty)
\end{equation}
(defined by analytic continuation).
It is an easy fact that for two neighboring direction classes $[\alpha_1],\, [\alpha_2]$ the  Stokes matrix  $S_{[\alpha_1][\alpha_2]}(\infty)$ is unipotent with only non-zero off-diagonal entries at the positions $(j,i)$ corresponding to the singularity $\lambda_i-\lambda_j$ separating the direction classes $[\alpha_1],\,[\alpha_2]$.

\begin{remark}
Note that by the Liouville-Ostrogradski formula $\det T_{[\alpha]}$ is constant in $z$ and therefore equal to 1.
\end{remark}

%
%\begin{small}
%\begin{remark}
%The bloc-solutions $\Psi_{[\alpha],\cdot j}(z)$ have the following geometric interpretation \cite{HLR}:
%The direction $\alpha$ divides the eigenvalues $\{\lambda_i\mid i\neq j\}$ in two parts according to whether they are on the left or right of the line $\lambda_j+e^{i\alpha}\R$.
%Let $W_{\alpha,j}^L$ be the space of the solutions of \eqref{eq:hg-2} that don't grow faster than $z^Ke^{-\frac{\lambda_j}{z}}$, for some $K<0$, when $z$ tends to 0 following a real positive trajectory of the vector field $e^{i(\alpha-\frac{\pi}{2})}z^2\partial_z$ (i.e. in direction $\alpha+\frac{\pi}{2}$) --- it corresponds to those $\lambda_i$ with 
%$\Re(e^{-i(\alpha+\frac{\pi}{2}}\lambda_i) \leq \Re(e^{i(\alpha+\frac{\pi}{2}}\lambda_j)$.
%Similarly, there is solution subspace $W_{\alpha,j}^R$ consisting of those solutions  of \eqref{eq:hg-2} that don't grow faster than $z^Ke^{-\frac{\lambda_j}{z}}$, for some $K<0$, when $z$ tends to 0 following a negative real trajectory of the vector field $e^{i(\alpha-\frac{\pi}{2})}z^2\partial_z$ (i.e. in direction $\alpha-\frac{\pi}{2}$) --- it corresponds to those $\lambda_i$ with $\Re(e^{-i(\alpha-\frac{\pi}{2}}\lambda_i) \leq \Re(e^{i(\alpha-\frac{\pi}{2}}\lambda_j)$.
%The two subspaces are transversal and their intersection $W_{\alpha,j}^L\cap W_{\alpha,j}^R$ is spanned by  $\Psi_{[\alpha],\cdot j}$. 
%\end{remark}
%\end{small}

\subsection{Fundamental solutions of the Okubo system}

The Okubo system \eqref{eq:hg-1} has $(p+1)$ Fuchsian singularities on $\CP^1$ at the points $\lambda_j$, $1\leq j\leq p$, and $\infty$.
Near each $s=\lambda_j$, the system is written as
\begin{equation*}
(s-\lambda_j)\frac{dv}{ds}=E_j(A+\rho)+{\cal O}(s-\lambda_j),
\end{equation*}
$E_j$ denoting the $j$-th column bloc of the identity matrix, 
and  ${\cal O}(s\!-\!\lambda_j)$ standing for holomorphic terms that vanish at $\lambda_j$.  
Its local ``multipliers" are therefore $A_{jj}+\rho$ in the $j$-th bloc and $0$ in the other $(p-1)$ blocs.

The system comes with two kinds of canonical mixed bases  that will be of interest in this article.
The first one, which will be denoted $V^+$,
consists of the so called \emph{Floquet solutions} $V_{\cdot j}^+(s,\rho)$, which behave asymptotically like $E_j(s-\lambda_j)^{A_{jj}+\rho}$ at the respective singularities $\lambda_j$.
The second one, denoted by $V^-$, is in a sense dual to the Floquet basis; it consists of solutions 
$V_{\cdot j}^-(s,\rho)$ that are analytic at each other singularity $\lambda_i$, $i\neq j$.
This section describes these two bases in more detail.

\begin{definition}
Let $P^+$ be a sector at $\infty$ in the parameter $\rho$-space, on which $\arg\rho\in\, ]\!-\!\pi+\eta,\pi-\eta[$, with $0<\eta<\frac{\pi}{2}$ fixed arbitrary,
and $|\rho|$ is sufficiently big  so that $\rho\notin -\N_{>0}-(\Spec A\cup\Spec A_D)$.

Symmetrically, let $P^-$ be a sector at $\infty$ on which $\arg\rho\in\, ]\eta,2\pi\!-\!\eta[$, and
$|\rho|$ is sufficiently big so that $\rho\notin \N_{>0}-(\Spec A\cup\Spec A_D)$.
\end{definition}

\paragraph{The Floquet bases.}

If $\Re(\rho)>0$ and $\rho$ is large enough so that all eigenvalues of $A_{jj}+\rho$ have positive real part,
then the matrix function $(s-\lambda_j)^{A_{jj}+\rho}$ vanishes when $s$ approaches $\lambda_j$ radially. 
Correspondingly, consider the space of solutions of \eqref{eq:hg-1} that vanish when $s\to \lambda_j$ radially.
\footnote{Note that no nontrivial combination of the other solutions corresponding to the multiplier $0$ can vanish at the singularity, they are asymptotically bigger and cannot hide behind the vanishing solutions. This is what makes this subspace of the space of solutions well-defined.}
It is invariant by the local monodromy, and 
it follows from the local theory of Fuchsian singularities (cf. \cite{IY, Le}) 
and the assumption \eqref{eq:hg-asum} on $A_{jj}$,
that this space has a unique basis written as
\begin{equation}\label{eq:hg-V+asympt} 
V_{\cdot j}^+(s,\rho)=(E_j+{\cal O}(s\!-\!\lambda_j))\cdot (s-\lambda_j)^{A_{jj}+\rho}.
\end{equation}
This construction can be extended to all parameters $\rho\in P^+$, if instead of letting  $s$ approach $\lambda_j$ radially, one lets it approach $\lambda_j$ following a suitable logarithmic spiral along which $(s-\lambda_j)^{A_{jj}+\rho}\to 0$.
More precisely, $s$ should follow a real positive trajectory of the vector field $-e^{i\theta}(s\!-\!\lambda_j)\partial_s$, for some
$\theta\in\, ]\!-\!\frac{\pi}{2},\frac{\pi}{2}[$ with $\Re(e^{i\theta} \rho)>0$.

The Floquet solution $V_{\cdot j}^+$ is closely related to the $j$-th formal canonical solution 
\eqref{eq:hg-Psi_alpha} of the dual Birkhoff system \eqref{eq:hg-2}. 
In fact, the formal Borel transform (=term-wise inverse Laplace transform) of $z^{\rho+1}\hat\Psi_{\cdot j}(z)$ equals
to the convolution integral \cite{Sch1}
\begin{equation}\label{eq:hg-RLI+}
I_{\cdot j}^+(s,\rho):=  %{\cal B}[z^{\rho+1}\hat\Psi_{\cdot j}(z)](s)=
\int_{\lambda_j}^{s} U_{\cdot j}(\sigma\!-\!\lambda_j) (s\!-\!\sigma)^{A_{jj}+\rho-1}d\sigma
\cdot\Gamma(A_{jj}\!+\!\rho)^{-1},
\end{equation}
where the matricial Gamma function is defined as usual by the integral 
$\Gamma(A_{jj}+\rho)=\int_0^{+\infty}t^{-(A_{jj}+\rho-1)}e^{-t}dt, \,$ 
%(cf. \cite{Gan}), 
and $U(s)$ is given in \eqref{eq:hg-U}.
The integral \eqref{eq:hg-RLI+}, also known as \emph{Riemann-Liouville integral} with base-point at $\lambda_j$, 
solves \eqref{eq:hg-1}, and moreover it satisfies 
$\,\frac{d}{ds}I_{\cdot j}^+(s,\rho)=I_{\cdot j}^+(s,\rho-1)$, and therefore solves the difference equation
\begin{equation*}
(s-B)I_{\cdot j}^+(s,\rho\!-\!1)=(A+\rho)I_{\cdot j}^+(s,\rho).
\end{equation*}
The canonical solution $\Psi_{[\alpha]}$ \eqref{eq:hg-Psi_alpha} of the Birkhoff system equals
\begin{equation} \label{eq:hg-I+Psi}
\Psi_{[\alpha],\cdot j}(z)=z^{-\rho-1}\int_{\lambda_j}^{+\infty e^{i\alpha}}I_{\cdot j}^+(s,\rho)\,e^{-\frac{s}{z}}ds.
\end{equation} 

The Floquet solution
$V_{\cdot j}^+$ is obtained from $I_{\cdot j}^+$ after a normalization:
\footnote{The fact that \eqref{eq:hg-V+} has the asymptotic behavior \eqref{eq:hg-V+asympt} is easily verified by integrating per partes.}
\begin{equation}\label{eq:hg-V+}
V_{\cdot j}^+(s,\rho)=\int_{\lambda_j}^{s} U_{\cdot j}(\sigma\!-\!\lambda_j) (s\!-\!\sigma)^{A_{jj}+\rho-1}d\sigma
%=\int_0^{s-\lambda_j} U_{\cdot j}(\sigma) (s\!-\!\lambda_j\!-\!\sigma)^{A_{jj}+\rho-1}d\sigma
\cdot(A_{jj}\!+\!\rho).
\end{equation} 
The integrating path in \eqref{eq:hg-V+} is such  that $\sigma$ follows a positive real trajectory of the vector field
$e^{i\theta}(s\!-\!\sigma)\partial_\sigma$ from the point $\lambda_j$ to $s$,  with suitable $\theta$ as above, avoiding other singularities $\lambda_i$, $i\neq j$, of $U_{\cdot j}(\sigma-\lambda_j)$.
The set of points $s$ which can be reached by such paths with varying $\theta$ then
defines a ramified domain on which \eqref{eq:hg-V+} is defined. 
Note that if $\arg(s-\lambda_j)=\alpha-\theta$, then the integrating trajectory approaches $\lambda_j$ in the asymptotic direction $\alpha$.

Let $\Omega^+_{[\alpha],j}(\rho)$ be a (ramified) domain consisting of those $s$ that can be reached by such trajectory for some direction $\alpha$ in the given homotopy class,
$$\Omega^+_{[\alpha],j}(\rho)\subseteq\{s\in\C\mid \arg(s-\lambda_j)=\alpha'-\theta,\ 
|\theta|<\tfrac{\pi}{2},\ |\theta+\arg\rho|<\tfrac{\pi}{2},\ \alpha'\in[\alpha]\},$$ 
and let $\Omega^+_{[\alpha]}(\rho):=\bigcap_j\Omega^+_{[\alpha],j}(\rho)$.
The restriction of $V^+$ to $\Omega^+_{[\alpha]}$ will be denoted $V_{[\alpha]}^+$.
Different homotopy classes of non-singular directions $[\alpha]$ give rise to to different branches $V_{[\alpha]}^+$ of $V^+$ near infinity.

\paragraph{The co-Floquet bases.}
For given index $j$, and a direction $\alpha$ such that $\lambda_i-\lambda_j\notin e^{i\alpha}\R^+$, $i\neq j$, define 
the co-Floquet solution $V_{[\alpha],\cdot j}^-$ at a singularity $\lambda_j$ as the unique solution analytic
on $\C\sminus \{\lambda_j+e^{i\alpha}\R^+\}$ and having the following asymptotic behavior near $\lambda_j$:
\begin{equation}\label{eq:hg-V-asympt} 
V_{[\alpha],\cdot j}^-(s,\rho)=(E_j+{O}(s\!-\!\lambda_j))\cdot (s\!-\!\lambda_j)^{A_{jj}+\rho},
\end{equation}
with $O$ denoting the usual Landau symbol (the corresponding terms may be ramified). 

Let's be more precise about where does it comes from.
For each singularity $\lambda_i$ and $\rho\in P^-$ large enough so that $A_{ii}+\rho$ has no positive eigenvalue,
%and therefore the matrix function $(s-\lambda_i)^{A_{ii}+\rho}$ explodes when $s\to\lambda_i$ following suitable logarithmic spiral,
define  $\Check W^-_{i}$ as the space of solutions analytic at $\lambda_i$. 
It follows from the local theory of Fuchsian singularities (cf. \cite{IY, Le}) that this space is tangent exactly to the the 
$(p-1)$ vector-blocs $E_k$, $k\neq i$, corresponding to the multiplier 0.
For a point $s\in \lambda_j-e^{i\alpha}\R^+ $, continue each solution subspace $\Check W^-_{i}$, $i\neq j$, toward $s$ in the cut plane $\C\sminus \{\lambda_j+e^{i\alpha}\R^+\}$, and
define the subspace $W_{\alpha,j}^-$ as their intersection. 
Since it consists of solutions analytic at each $\lambda_k$, $k\neq j$, 
it does not depend on the way the $\Check W_{i}^-$ are continued around the singularities $\lambda_k$, only on the direction $\alpha$ of the cut. 

Following \cite{Sch1}, there is a canonical bloc-solution of \eqref{eq:hg-1} generating the space $W_{\alpha,j}^-$ 
given by the integral:
\begin{equation}\label{eq:hg-RLI-}
I_{[\alpha]}^-(s,\rho):=\int_{\lambda_j}^{+\infty e^{i\alpha}} U_{\cdot j}(\sigma\!-\!\lambda_j) (s\!-\!\sigma)^{A_{jj}+\rho-1}d\sigma \cdot\Gamma(1\!-\!A_{jj}-\rho)e^{-\pi i (A_{jj}+\rho)},
\end{equation}
which satisfies again
$\,\frac{d}{ds}I_{\cdot j}^-(s,\rho)=I_{\cdot j}^-(s,\rho-1)$, and therefore solves the difference equation
\begin{equation}\label{eq:hg-I-}
(s-B)I_{\cdot j}^-(s,\rho\!-\!1)=(A+\rho)I_{\cdot j}^-(s,\rho).
\end{equation} 
The integral $I_{[\alpha]}^-$ is a Laplace transform of the canonical solution $\Psi_{[\alpha]}$
\begin{equation*}
I_{[\alpha]}^-(s,\rho):=\int_{0}^{+\infty e^{i\alpha}} z^{\rho-1}\Psi_{[\alpha],\cdot j}(z) e^{\frac{s}{z}} dz,
\end{equation*}
which in turn equals to 
\begin{equation}\label{eq:hg-I-Psi}
\Psi_{[\alpha],\cdot j}(z)=z^{-\rho-1} \frac{1}{2\pi i} \int_{\gamma_{j,\alpha}} I_{\cdot j}^-(s,\rho)\,e^{-\frac{s}{z}}ds,
\end{equation} 
where the path $\gamma_{j,\alpha}$ encircles the ray $\lambda_j+e^{i\alpha}\R^+$ in positive direction.
While $\Psi_{[\alpha]}$ is defined on a sector at 0 of an opening $>\pi$ bisected by $e^{i\alpha}\R^+$,
the integral $I_{[\alpha]}^-$ is defined on a sector at $\lambda_j$ bisected by $\lambda_j+e^{i(\alpha+\pi)}\R^+$ of an opening $>2\pi$.

The co-Floquet solution is obtained after a normalization
\begin{equation}\label{eq:hg-V-}
V_{[\alpha],\cdot j}^-(s,\rho)=\int_{\lambda_j}^{\infty} U_{\cdot j}(\sigma\!-\!\lambda_j) (s\!-\!\sigma)^{A_{jj}+\rho-1}d\sigma
\cdot(A_{jj}\!+\!\rho).
\end{equation}
In the default situation when $\Re(\rho)<0$ the integration path is the straight ray
$\sigma\in\lambda_j+e^{i\alpha}\R^+$ and the integral is defined for $s\in\lambda_j-e^{i\alpha}\R^+$ and extended analytically from there.
In a general situation, the integration path follows a negative real trajectory of the vector field
$e^{i\theta}(s\!-\!\sigma)\partial_\sigma$ from the point $\lambda_j$ to $\infty$, with a suitable 
$\theta\in\, ]\!-\!\frac{\pi}{2},\frac{\pi}{2}[$ such that $\Re(e^{i\theta} \rho)<0$, 
that is end-point homotopic to the ray $\lambda_j+e^{i\alpha}\R^+$ in
$\CP^1\sminus\{\lambda_i,\ i\neq j\}$.
The set of points $s$ that can be reached by such paths 
%(with varying $|\theta|<\frac{\pi}{2}$ and direction $\alpha$ in its homotopy class $[\alpha]$)
defines again a (ramified) sectoral domain 
$$\Omega^-_{[\alpha],j}(\rho)\subseteq\{s\in\C\mid \arg(s-\lambda_j)=\alpha'-\theta+\pi,\ 
|\theta|<\tfrac{\pi}{2},\ |\theta-\pi+\arg\rho|<\tfrac{\pi}{2}, \alpha'\in[\alpha]\},$$ 
on which the integral \eqref{eq:hg-V-} is naturally defined. Let $\Omega^-_{[\alpha]}(\rho):=\bigcap_j\Omega^-_{[\alpha],j}(\rho)$.

\begin{proposition}\label{prop:hg-1}
For $\rho\in P^+\cap P^-$ and $s\in \lambda_j+e^{i \alpha}\R^+$, let $\tilde s=\lambda_j+e^{2\pi i}(s-\lambda_j)$, then
\begin{equation*}
V_{[\alpha],\cdot j}^+(s,\rho)=\big[V_{[\alpha],\cdot j}^-(\tilde s,\rho)-V_{[\alpha],\cdot j}^-(s,\rho)\big]\cdot \big[e^{2\pi i(A_{jj}+\rho)}-1 \big]^{-1},
\end{equation*}
or equivalently
\begin{equation*}
I_{[\alpha],\cdot j}^+(s,\rho)=\frac{1}{2\pi i}\big[I_{[\alpha],\cdot j}^-(\tilde s,\rho)-I_{[\alpha],\cdot j}^-(s,\rho)\big],
\end{equation*}
i.e. $I_{[\alpha],\cdot j}^+$ is a hyperfunction defined by the boundary value of $\frac{1}{2\pi i}I_{[\alpha],\cdot j}^-$ on $\lambda_j+e^{i \alpha}\R^+$.
\end{proposition}

\begin{proof}
Follows from the construction.
\end{proof}

The following Proposition is due to Okubo and Kohno.

\begin{proposition}[Gauss'--Kummer's formula]~
\begin{align}
\det V_{[\alpha]}^+(s,\rho)&=\frac{\det \Gamma(A_D+\rho+1)}{\det\Gamma(A+\rho+1)}\cdot\det\,(s-B)^{A_D+\rho},\\[8pt]
\det V_{[\alpha]}^-(s,\rho)&=\frac{\det \Gamma(-A-\rho)}{\det\Gamma(-A_D-\rho)}\cdot\det\,(s-B)^{A_D+\rho}.
\end{align}
\end{proposition}

\begin{proof}%[Sketch of proof]
For the sake of completeness we will sketch here the proof in the co-Floquet case; the Floquet case is almost identical
and can be found in \cite{Ok1, Ko1, Ko2}.

The co-Floquet solution has the following asymptotic behavior w.r.t. $\rho$  (see \cite{Sch1}, theorem (4.6)):

\begin{equation}\label{eq:hg-V-as} 
V_{[\alpha],\cdot j}^-(s,\rho)\cdot (s\!-\!\lambda_j)^{-A_{jj}-\rho}=E_j+{O}(\frac{1}{|\rho|}),\quad\text{when }\ \Re(\rho)\to -\infty,
\end{equation}
locally uniformly in the cut plane $\C\sminus(\lambda_j+e^{i\alpha}\R^+)$.
%Let $\rho_m:=\rho\!-\!m$, , and let $m\to+\infty$.
Now, for $m\in\N$ it follows from \eqref{eq:hg-I-} by induction that
$$I_{[\alpha],\cdot j}^-(s,\rho\!-\!m)\cdot (s\!-\!\lambda_j)^{m}=
(A\!+\!\rho\!-\!m\!+\!1)\cdots(A\!+\!\rho)\cdot I_{[\alpha],\cdot j}^-(s,\rho),$$
and hence
\begin{align*}
\Gamma(-A\!-\!\rho)&\cdot\Gamma(-A\!-\!\rho\!+\!m)^{-1}\cdot 
 V_{[\alpha],\cdot j}^-(s,\rho-m)\cdot (s\!-\!\lambda_j)^{-A_{jj}-\rho+m}=\\
&\qquad = V_{[\alpha],\cdot j}^-(s,\rho) \cdot (s\!-\!\lambda_j)^{-A_{jj}-\rho}\cdot 
\Gamma(-A_{jj}\!-\!\rho)\cdot\Gamma(-A_{jj}\!-\!\rho+\!m)^{-1}.
\end{align*}
Therefore
\begin{align*}
\frac{\det\Gamma(-A_D\!-\!\rho\!+\!m)}{\det\Gamma(-A\!-\!\rho\!+\!m)}
&\cdot \det\left[V_{[\alpha]}^-(s,\rho-m)\cdot (s\!-\!B)^{-A_{D}-\rho+m}\right]=\\
&=\det\left[V_{[\alpha]}^-(s,\rho)\cdot (s\!-\!B)^{-A_{D}-\rho}\right]\cdot
\frac{\det\Gamma(-A_D\!-\!\rho)}{\det\Gamma(-A\!-\!\rho)}.
\end{align*}
Letting $m\to+\infty$ and using \eqref{eq:hg-V-as} and usual formulas for the $\Gamma$-function, one can see that
both expressions on the left side tend to 1.
\end{proof}

\begin{corollary}
For $\rho\in P^+$ (resp. $\rho\in P^-$) the Floquet (resp. the co-Floquet) solutions form  a basis of the solution space.
\end{corollary}

%
%\begin{small}
%\begin{remark}
%Let us stress out the symmetric nature of the two bases.
%For $i\neq j$, the change of the integration variable
%$s-\sigma=\frac{(s-\lambda_i)(\tau-\lambda_j)}{(\tau-\lambda_i)}$, 
%transforming the vector field $e^{i\theta}(s\!-\!\sigma)\partial_\sigma$ to
%$-e^{i\theta}\frac{(\tau-\lambda_i)(\tau-\lambda_j)}{(\lambda_i-\lambda_j)}\partial_\tau$, 
%allows to write
%\begin{align*}
%V_{[\alpha],ij}^\pm(s,\rho)=&\int_s^{\lambda_{ij}^\pm} U_{ij}\Big((\lambda_i\!-\!\lambda_j)(1-\tfrac{s-\lambda_i}{\tau-\lambda_j})\Big)\Big(\frac{\tau\!-\!\lambda_j}{\tau\!-\!\lambda_i}\Big)^{A_{jj}+\rho}
%\frac{(\lambda_i\!-\!\lambda_j) \,d\tau}{(\tau\!-\!\lambda_i)(\tau\!-\!\lambda_j)} \\
%&\cdot(A_{jj}\!+\!\rho)\,(s\!-\!\lambda_j)^{A_{jj}+\rho} ,
%\end{align*} 
%with $\lambda_{ij}^+=\lambda_{j}$ and $\lambda_{ij}^-=\lambda_{i}$.
%The diagonal elements $V_{[\alpha],jj}^\pm(s,\rho)$ are completely determined by the non-diagonal ones through
%\eqref{eq:hg-1} and \eqref{eq:hg-V+asympt}, \eqref{eq:hg-V-asympt},
%\begin{equation*}
%V_{[\alpha],jj}^\pm(s,\rho)=(s\!-\!\lambda_j)^{A_{jj}+\rho}+
%\int_{\lambda_j}^s \Big(\frac{s\!-\!\lambda_j}{\tau\!-\!\lambda_j}\Big)^{A_{jj}+\rho}
%\sum_{i\neq j}A_{ji}V_{[\alpha],ij}^\pm(\tau,\rho) \frac{\,d\tau}{(\tau\!-\!\lambda_j)}.
%\end{equation*} 
%\end{remark}
%\end{small}

\subsection{Fundamental matrix solutions of the confluent family}

The system \eqref{eq:hg-3} has two kinds of canonical fundamental matrix solutions $Y_{[\alpha]}^\pm(z,\rho)$ corresponding to the Floquet 
and co-Floquet  bases of \eqref{eq:hg-2}.
In order to obtain a convergence when $\rho\to\infty$, one has to be a bit careful with the choice of their branch.
It is convenient to write them as
\begin{equation}\label{eq:hg-Y}
Y_{[\alpha]}^\pm(z,\rho)=T_{[\alpha]}^\pm(z,\rho)\cdot \Phi(z,\rho),
\end{equation}
where
$$\Phi(z,\rho):=z^{-\rho}(z-\tfrac{B}{\rho})^{A_{D}+\rho},$$ 
is a canonical solution to the bloc-diagonal system
\begin{equation}\label{eq:hg-3formal}
z(z-\tfrac{B}{\rho})\frac{d\phi}{dz}=(B+zA_D)\phi,
\end{equation} 
whose branch needs to be selected so that it converges to the adequate branch of 
$$\Phi(z,\infty):=z^{A_{D}} e^{-\frac{B}{z}},$$ 
when $\rho\to\infty$.
The bloc-diagonalizing transformation $T_{[\alpha]}^\pm$ is defined by  
$$T_{[\alpha]}^\pm(z,\rho)=V_{[\alpha]}^\pm(\rho z,\rho)\cdot(\rho z-B)^{-A_{D}-\rho},$$
where the branch of $(\rho z-B)^{-A_{D}-\rho}$ is chosen in accord with the one inside the integral \eqref{eq:hg-V+}, \eqref{eq:hg-V-}. 
Hence
\begin{equation}\label{eq:hg-T+}
T_{[\alpha],\cdot j}^+(z,\rho)=\int_0^{\rho z-\lambda_j} U_{\cdot j}(\sigma) \Big(\frac{\rho z-\lambda_j-\sigma}{\rho z-\lambda_j}\Big)^{A_{jj}+\rho}\frac{d\sigma}{\rho z-\lambda_j-\sigma} \cdot(A_{jj}+\rho),
\end{equation} 
where the integration path follows a positive trajectory of the vector field $\,e^{i\theta}(\rho z\!-\!\lambda_j\!-\!\sigma)\partial_\sigma$ from the point 0 to $\rho z\!-\!\lambda_j$,
and the branch of $\Big(\frac{\rho z-\lambda_j-\sigma}{\rho z-\lambda_j}\Big)^{A_{jj}+\rho}$ is chosen so that it is equal 1  at the endpoint.
Remark, that at the limit, when $\rho\to\infty$ radially with fixed $\arg\rho$, 
the trajectories of the given vector field become trajectories of the vector field $e^{i\alpha}\partial_z$ with 
$\alpha=\theta+\arg\rho+\arg z$.
Therefore the integral \eqref{eq:hg-T+} has a well-defined limit
\begin{equation*}
T_{[\alpha],\cdot j}^+(z,\infty)=\frac{1}{z}\int_0^{+\infty e^{i\alpha}} U_{\cdot j}(\sigma)\, e^{-\frac{\sigma}{z}}
d\sigma.
\end{equation*} 

\bigskip
Similarly,
\begin{equation}\label{eq:hg-T-}
T_{[\alpha],\cdot j}^-(z,\rho)=\int_0^{\infty} U_{\cdot j}(\sigma) \Big(\frac{\rho z-\lambda_j-\sigma}{\rho z-\lambda_j}\Big)^{A_{jj}+\rho}\frac{d\sigma}{\rho z-\lambda_j-\sigma} \cdot(A_{jj}+\rho),
\end{equation} 
where the integration path follows a positive trajectory of  $\,e^{i\theta-\pi}(\rho z\!-\!\lambda_j\!-\!\sigma)\partial_\sigma$ from the point 0 to $\infty$,
which at the limit, when $\rho\to\infty$ radially, becomes 
a trajectory of $\,e^{i\alpha}\partial_z$ with 
$\alpha=\theta+\arg\rho-\pi+\arg z$, 
and the integral \eqref{eq:hg-T-} becomes
\begin{equation*}
T_{[\alpha],\cdot j}^-(z,\infty)=\frac{1}{z}\int_0^{+\infty e^{i\alpha}} U_{\cdot j}(\sigma)\, e^{-\frac{\sigma}{z}}
d\sigma.
\end{equation*} 

The transformations $T_{[\alpha]}^\pm(\cdot,\rho)$ are defined on sectors
$$\mscr S_{[\alpha]}^\pm(\rho):=\tfrac{1}{\rho}\Omega^\pm_{[\alpha]}(\rho),$$
which tend to a subsector of $\mscr S_{[\alpha]}(\infty)$ \eqref{eq:hg-sector} depending on the radial direction in which $\rho\to\infty$.

\paragraph{Stokes matrices of the confluent family}
Fixing a branch of $\Phi(z,\rho)$ near $z=\infty$ and its restrictions to the sectors
$\mscr S_{[\alpha]}^\pm(\rho)$ one obtains a canonical set of fundamental matrix 
solutions $Y_{[\alpha]}^\pm(z,\rho)$ \eqref{eq:hg-Y}.
The connection matrices between these solutions near $z=\infty$ corresponding to different non-singular directions $\alpha_1,\, \alpha_2$
$$Y_{[\alpha_2]}^\pm=Y_{[\alpha_1]}^\pm\cdot S_{[\alpha_1][\alpha_2]}^\pm(\rho)$$
will be called \emph{Stokes matrices} of the family. 

\begin{proposition}\label{prop:hg-stokesmatrices}
\begin{align*}
S_{[\alpha_1][\alpha_2]}^+(\rho)&=\rho^{A_D}\Gamma(A_D+\rho+1)^{-1}
S_{[\alpha_1][\alpha_2]}(\infty)\,\Gamma(A_D+\rho+1)\rho^{-A_D},
\\
S_{[\alpha_1][\alpha_2]}^-(\rho)&=\big(e^{-\pi i}\rho\big)^{-A_D}\Gamma(-A_D-\rho)\,
S_{[\alpha_1][\alpha_2]}(\infty)\,\Gamma(-A_D-\rho)^{-1}\big(e^{-\pi i}\rho\big)^{-A_D},
\end{align*}
which tends to the Stokes matrix $S_{[\alpha_1][\alpha_2]}(\infty)$ \eqref{eq:hg-stokesmatrices} when $\rho\to\infty$ in
$P^\pm$ respectively.
For two neighboring direction classes $[\alpha_1],\, [\alpha_2]$ the  Stokes matrix  $S_{[\alpha_1][\alpha_2]}^\pm(\rho)$ is unipotent with only non-zero 
off-diagonal blocs at the positions $(i,j)$ corresponding to the direction of $(\lambda_i-\lambda_j)\R+$ separating $\alpha,\,\alpha'$.
\end{proposition}

\begin{proof}
This follows from the relation
$$I_{[\alpha_2]}^\pm(s,\rho)=I_{[\alpha_1]}^\pm(s,\rho)\cdot S_{[\alpha_1][\alpha_2]}^\pm(\rho),$$
which is a consequence of the formulas \eqref{eq:hg-I+Psi}, resp.  \eqref{eq:hg-I-Psi}.
\end{proof}

\begin{figure}[t]
\centering
\subfigure[$\rho\in\R+$]{\includegraphics[width=0.47\textwidth]{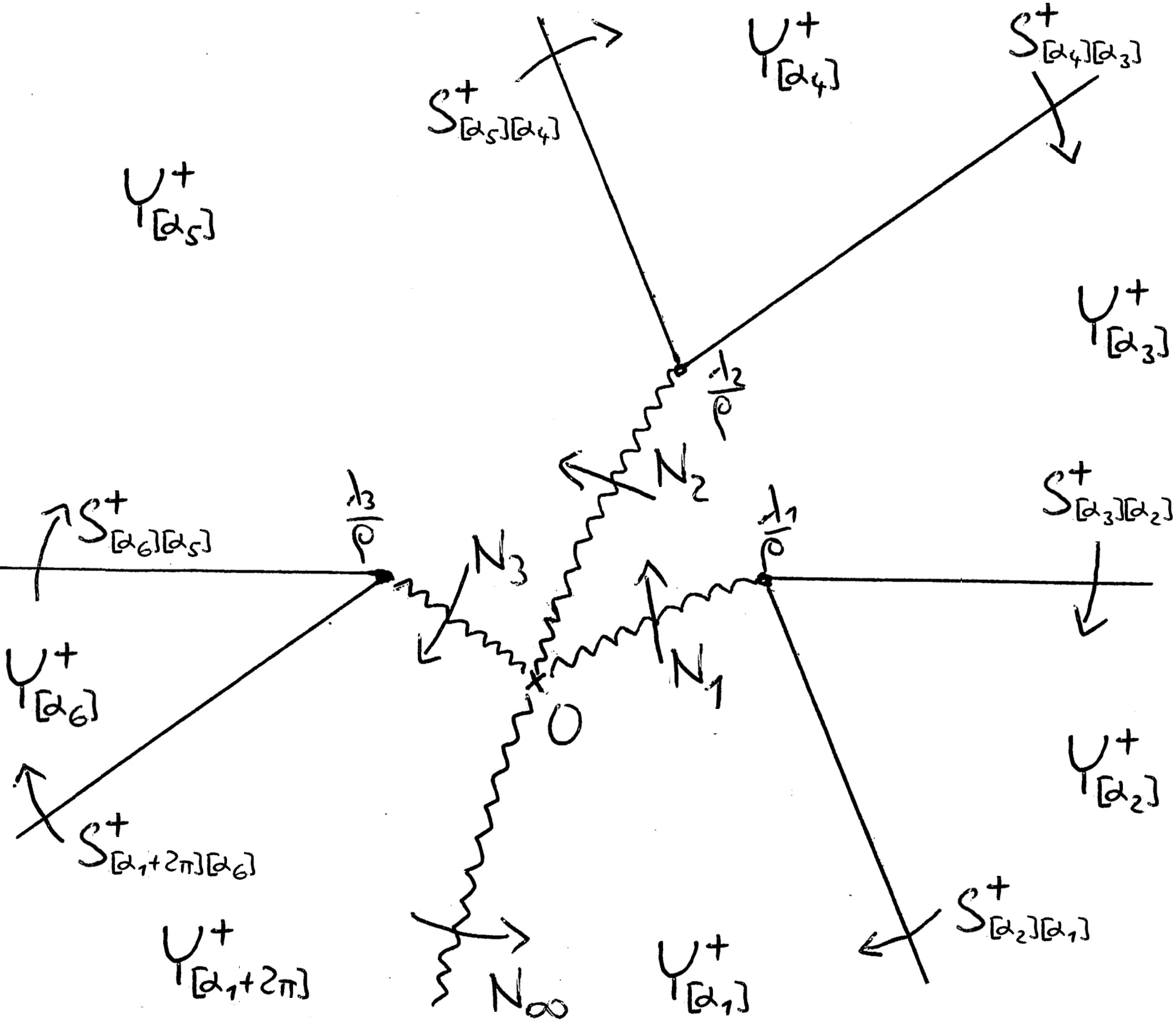}}
\hskip 0.05\textwidth
\subfigure[$\rho\in\R-$]{\includegraphics[width=0.47\textwidth]{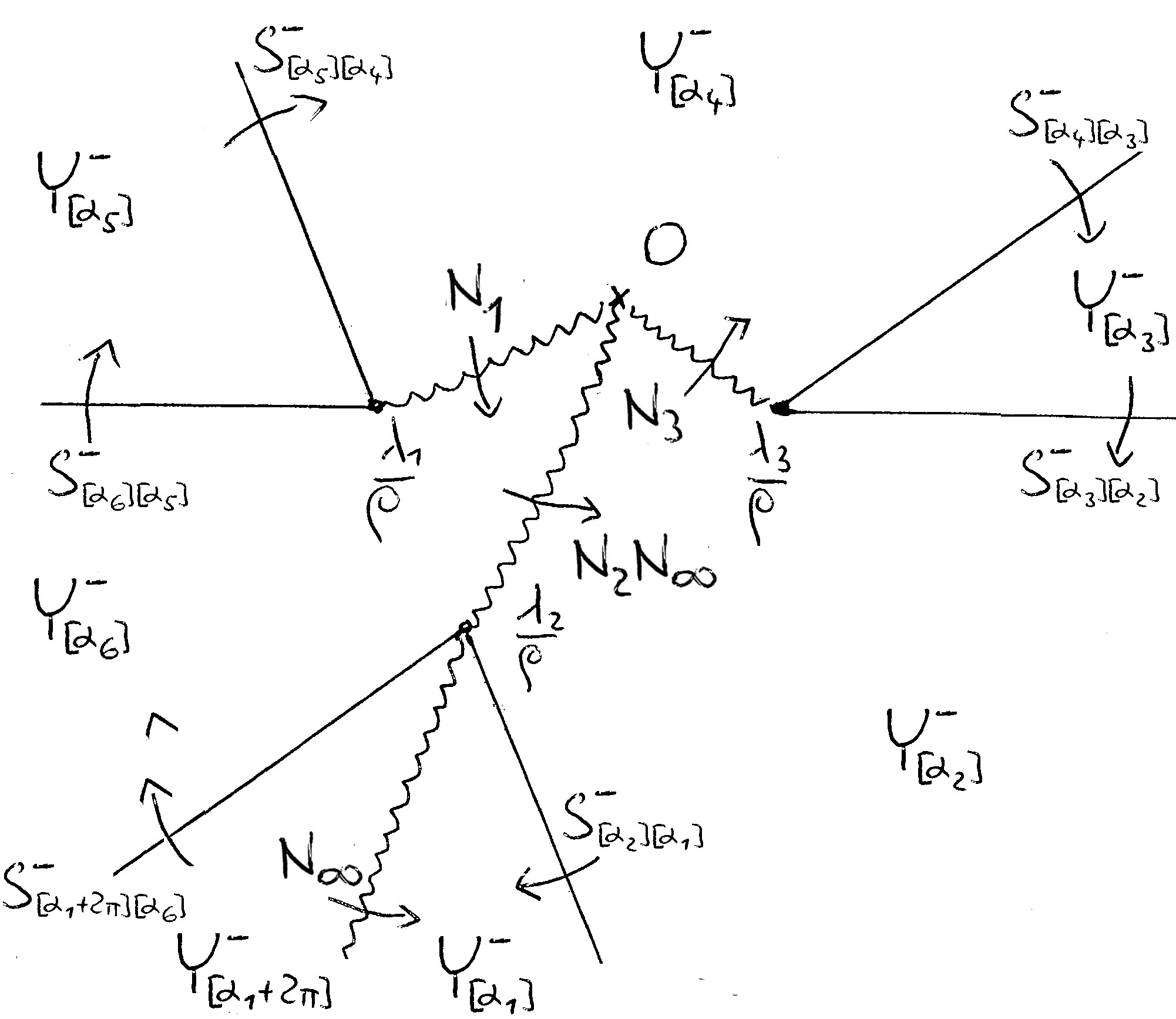}}
\caption{The connection matrices between different branches of $Y^\pm$ in Example~\ref{example:hg-1}.}
\label{figure:hg-1}
\end{figure}

\begin{example}[Figure~\ref{figure:hg-1}]\label{example:hg-1}
Suppose $B$ has just three eigenvalues $\lambda_i$, $i=1,2,3$ and assume they are not colinear.
For simplicity we will consider only the default situation when $\rho\in\R\pm$ and restrict the domains of $Y_{[\alpha]}^\pm$ to a smaller
sector consisting of the points $z\in\C$ for which the integration path in \eqref{eq:hg-T+}, resp. \eqref{eq:hg-T-}, can be taken a straight segment.
Near $\infty$, these sectors are are separated by the outer parts of lines through $\lambda_i,\lambda_k$, whose crossing is governed by the Stokes matrices.
For each singularity  $\frac{\lambda_j}{\rho}$ or $\infty$ make a cut (wavy line in Figure~\ref{figure:hg-1}) from the origin on which the formal solution $\Phi$ is branched, and therefore changed by its formal monodromy $$N_j=\left(\begin{smallmatrix}I_{n_1}\hskip-3pt &&&& \\[-6pt] & \ddots &&& \\[1pt] &&\hskip-12pt e^{2\pi i(A_{jj}+\rho)}\hskip-18pt && \\[-4pt] &&& \ddots & \\ &&&& I_{n_3}\end{smallmatrix}\right),\quad j=1,2,3,
\quad\text{and}\quad N_\infty=e^{-2\pi i A_{D}}.$$

\end{example}

\goodbreak

\section{Confluence in generalized hypergeometric equation}

This section illustrates the confluence procedure on the example of the generalized hypergeometric equation, where things can be expressed very explicitly.
Most of the formulas come from \cite{Du, OTY, Lu}. To simplify the writing we adopt the following notation.

\begin{notation}
Let $\bm\alpha=(\alpha_1,\ldots,\alpha_n)$, $\bm\beta=(\beta_1,\ldots,\beta_n)\in\C^n$ and
denote 
\begin{itemize}
\item $\bmbeta{\,j}\in\C^{n-1}$ obtained from $\bm\beta$ by omitting the $j$-th  component, similarly with
$\bmbeta{n,j}\in\C^{n-2}$,
\item for $c\in\C$, let $\,\bm\alpha-c:=(\alpha_1-c,\ldots,\alpha_n-c),$
\item for a function $f:\C\to\C$, write shortly 
 $\,f(\bm\alpha):=\displaystyle\Pi_{i}f(\alpha_i).$
\end{itemize}
\end{notation}

In the above notation, \emph{the generalized hypergeometric equation} of order $n$, is written as $\displaystyle D\hg{\bm\alpha}{\bm\beta}{s}w=0$, where
\begin{align}\label{eq:hg-hg}
D\hg{\bm\alpha}{\bm\beta}{s} 
:&=s(\delta_s+\alpha_n)\ldots(\delta_s+\alpha_1)-
(\delta_s+\beta_{n}\!-\!1)\ldots(\delta_s+\beta_1\!-\!1) \nonumber \\ 
&=s(\delta_s+\bm\alpha)-(\delta_s+\bm\beta-1),
\end{align}
and $\delta_s=s\frac{d}{ds}$ is the Euler operator. 
It has three regular singular points at $0$, $1$ and $\infty$.

Since 
\begin{equation}\label{eq:hg-D1}
D\hg{\bm\alpha}{\bm\beta}{s}(s^c w)=s^cD\hg{\bm\alpha\!+\!c}{\bm\beta\!+\!c}{s}w,
\end{equation}
one can always use the transformation $w\mapsto s^{1-\beta_j} w$, to bring the equation to a more
usual form in which one of the $\beta_j$'s equals 1: 
$\displaystyle D\hg{\bm\alpha\!+\!1\!-\!\beta_j}{\bm\beta\!+\!1\!-\!\beta_j}{s}$.
The equation \eqref{eq:hg-hg} has thus $n$ local solution at $s=0$ given by  
the hypergeometric series   
$$s^{1-\beta_j} {}_nF_{n-1}\hg{\bm\alpha\!+\!1\!-\!\beta_j}{\bmbeta{j}\!+\!1\!-\!\beta_j}{s}:=
s^{1-\beta_j} \sum_{k=0}^{+\infty}\frac{(\bm\alpha\!+\!1\!-\!\beta_j)_k}{(\bmbeta{j}\!+\!1\!-\!\beta_j)_k(1)_k}s^k, \quad 1\leq j\leq n,$$
convergent for $|s|<1$,
where $(a)_k$ denotes the Pochhammer symbol 
$$(a)_k=a\,(a\!+\!1)\ldots(a\!+\!k\!-\!1),\quad  (a)_0=1.$$
These solutions are linearly independent if no two $\beta_j$'s differ by an integer.

\begin{remark}
There is a symmetry between the singular points $0$ and $\infty$ given by the relation
\begin{equation}\label{eq:hg-D2}
D\hg{\bm\alpha}{\bm\beta}{s}= (-1)^ns\, D\hg{1\!-\!\bm\beta}{1\!-\!\bm\alpha}{\frac{1}{s}}.
\end{equation}

In the case of the Gauss' hypergeometric equation ($n=2$)
there is also a symmetry between the two singular points $0,1$
due to the relation 
$$D\hg{\alpha_1,\alpha_2}{1,\beta}{s}=\frac{s}{1\!-\!s}D\hg{\alpha_1,\alpha_2}{1,-\gamma}{1\!-\!s},\qquad \gamma=\beta\!-\!1\!-\!\alpha_1\!-\!\alpha_2.$$
This symmetry is broken for $n>2$.
\end{remark}

\paragraph{The confluence.}
We are interested in the situation when $\beta_n\to\infty$. The situation when $\alpha_n\to\infty$ would be similar due to the symmetry \eqref{eq:hg-D2}.
Let 
$$\rho=\beta_n-1,\qquad s=\rho z$$ 
then 
\begin{equation}\label{eq:hg-hgco}
\tfrac{1}{\rho} D\hg{\bm\alpha}{\bm\beta}{\rho z}=z(\delta_z+\bm\alpha)-(\tfrac{1}{\rho}\delta_z+1)(\delta_z+\bmbeta{n}-1),
\end{equation}
where the regular singularities at $z=0$ and $z=\frac{1}{\rho}$ merge for $\rho\to\infty$ to form an irregular
singularity.

Setting
$y=(y_1,\ldots,y_n)^\msf{T}$, with
$$y_{i+1}(z,\rho)=(\delta_z+\beta_i\!-\!1)y_i(z,\rho),\quad\text{for}\quad 1\leq i\leq n\!-\!1,$$
the equation \eqref{eq:hg-hgco}: $\,\displaystyle\tfrac{1}{\rho} D\hg{\bm\alpha}{\bm\beta}{\rho z}y_1=0\,$
is  written in the form of a family of systems \eqref{eq:hg-3} with
\begin{equation}\label{eq:hg-AB}
B=\begin{pmatrix}
0 & & & \\[-3pt] 
& \ddots & & \\
& & 0 & \\ 
& & & 1
\end{pmatrix},\qquad
A=\begin{pmatrix}
1\!-\!\beta_1\hskip-9pt  & \hskip6pt 1\hskip-6pt  & & \\[-3pt]
& \ddots & \ddots  & \\
& & \hskip-12pt 1\!-\!\beta_{n-1} \hskip-6pt & 1 \\
* & \ldots & * & \gamma
\end{pmatrix},
\end{equation}
$$\gamma:=\sum_1^{n-1}(\beta_j\!-\!1)-\sum_1^{n}\alpha_n,$$
where $A$ has $-\alpha_1,\ldots,-\alpha_n$ as eigenvalues.

\paragraph{The Okubo system.}
The corresponding Okubo system 
\begin{equation}\label{eq:hg-1x} 
(s-B)\frac{dv}{ds}=(A+\rho)v
\end{equation}
 for $v(s,\rho)=s^{-\rho}y(z,\rho)$, 
 is associated to the generalized hypergeometric equation
$$
\displaystyle D\hg{\bm\alpha\!-\!\rho}{\bm\beta\!-\!\rho}{s}v_1(s,\rho)=0,
\qquad 
v_{i+1}(s,\rho)=(\delta_s+\beta_i\!-\!1\!-\!\rho)v_i(s,\rho). 
$$ 

Suppose now, that
\begin{equation}\label{eq:hg-asump2}
\beta_i-\beta_j\notin\Z\ \text{ for all } i\neq j.
%, \quad\text{and}\quad \alpha_i-\beta_j\notin\Z_{<0}\ \text{ for all }i,j.
\end{equation}

For $\rho\neq\infty$, we have $n-1$ singular solutions of the Okubo system near $s=0$ whose first component is given by
\begin{equation*}
\tilde v_{1j}^+(s,\rho)
=s^{1-\beta_j+\rho}\, {}_nF_{n-1}\hg{\bm\alpha\!+\!1\!-\!\beta_j}{\bmbeta{j}\!+\!1\!-\!\beta_j}{s}, \qquad |s|<1,\quad 1\leq j\leq n-1,
\end{equation*}
and one singular solution at $s=1$ given by Meijer G-function
\begin{align*}
\tilde v_{1n}^+(s,\rho)
&= G_{n,n}^{n,0}\hg{\bm\beta}{\bm\alpha}{\frac{1}{z}} \cdot\Gamma(1\!+\!\gamma\!+\!\rho),\qquad |s|>1,\notag \\
&=(s\!-\!1)^{\gamma+\rho}\sum_{k=0}^{+\infty}\frac{(-1)^k c_k}{(\gamma+\rho+1)_{k+n-1}}(s\!-\!1)^{k+n-1},\qquad |s-1|<1,
\end{align*}
with $c_0=1$ and the coefficients $c_k$ independent of $\rho$ (see \cite{Du}, p.~601)
\begin{equation}\label{eq:hg-ck}
\,c_k=\text{\footnotesize{$\displaystyle\sum_{i_1+\ldots+i_{n-1}=k}
\prod_{j=1}^{n-1}\frac{
(\beta_1\!-\alpha_1\!+i_1+\ldots+\beta_j\!-\alpha_j\!+i_j)_{i_j}\cdot (\beta_j\!-\!\alpha_{j+1})_{i_j}}{i_j!}$}}.
\end{equation}

It is easy to see that the terms of the fundamental solution matrix 
$\tilde V^+=(\tilde v_{ij}^+)$ have the asymptotic behavior
$$\tilde V^+(s,\rho)\sim \left(R+O(1)(s-B)\right)(s-B)^{\tilde A_D+\rho},$$
where the upper-triangular matrix $R=(r_{ij})$
$$r_{1j}=1,\quad r_{ij}=(\beta_1\!-\!\beta_j)\ldots(\beta_{i-1}\!-\!\beta_j),\ i>1,\quad\text{and}\quad
r_{nn}=1,\quad r_{in}=0,\  i< n ,$$
commutes with $B$ and diagonalizes 
\begin{equation*}%\label{eq:hg-1}
A_D=\begin{pmatrix}
1\!-\!\beta_1\hskip-9pt  & \hskip3pt 1\hskip-6pt  & & &\\[-0.7\normalbaselineskip] 
& \raisebox{0.15\normalbaselineskip}{$\ddots$} & \hskip-3pt\raisebox{0.7\normalbaselineskip}{$\ddots$}  &\hskip-12pt 1 &\\
 & & & \hskip-24pt 1\!-\!\beta_{n-1} \hskip-6pt &  \\
 & &  &  & \,\gamma
\end{pmatrix},
\qquad
R^{-1}\!A_DR=\tilde A_D:=\begin{pmatrix}
1\!-\!\beta_1\hskip-9pt  & \hskip3pt   & & \\[-3pt] & \ddots &   & \\ & & \hskip-12pt 1\!-\!\beta_{n-1} \hskip-6pt &  \\  &  &  & \,\gamma
\end{pmatrix}.
\end{equation*}
Therefore,
\begin{equation}\label{eq:hg-VV}
V^{+\!}(s,\rho)=\tilde V^{+\!}(s,\rho)R^{-1}
\end{equation}
is the Floquet bases of the Okubo system \eqref{eq:hg-1x} with the asymptotic behavior \eqref{eq:hg-V+asympt},
while $R^{-1}\tilde V^+$ is the Floquet basis of the Okubo system with
$\tilde B=B$, $\tilde A=R^{-1}AR$.
%=\left(\begin{smallmatrix}
%1\!-\!\beta_1\hskip-6pt  & \hskip3pt   & & *\\[-4pt] 
%& \ddots &   & \vdots\\[2pt] 
%& & \hskip-6pt 1\!-\!\beta_{n-1} \hskip-3pt & * \\[2pt]
% * & \lgridwidth=dots & * & \,\gamma
%\end{smallmatrix}\right)

\begin{figure}[t]
\centering
\includegraphics[width=\textwidth]{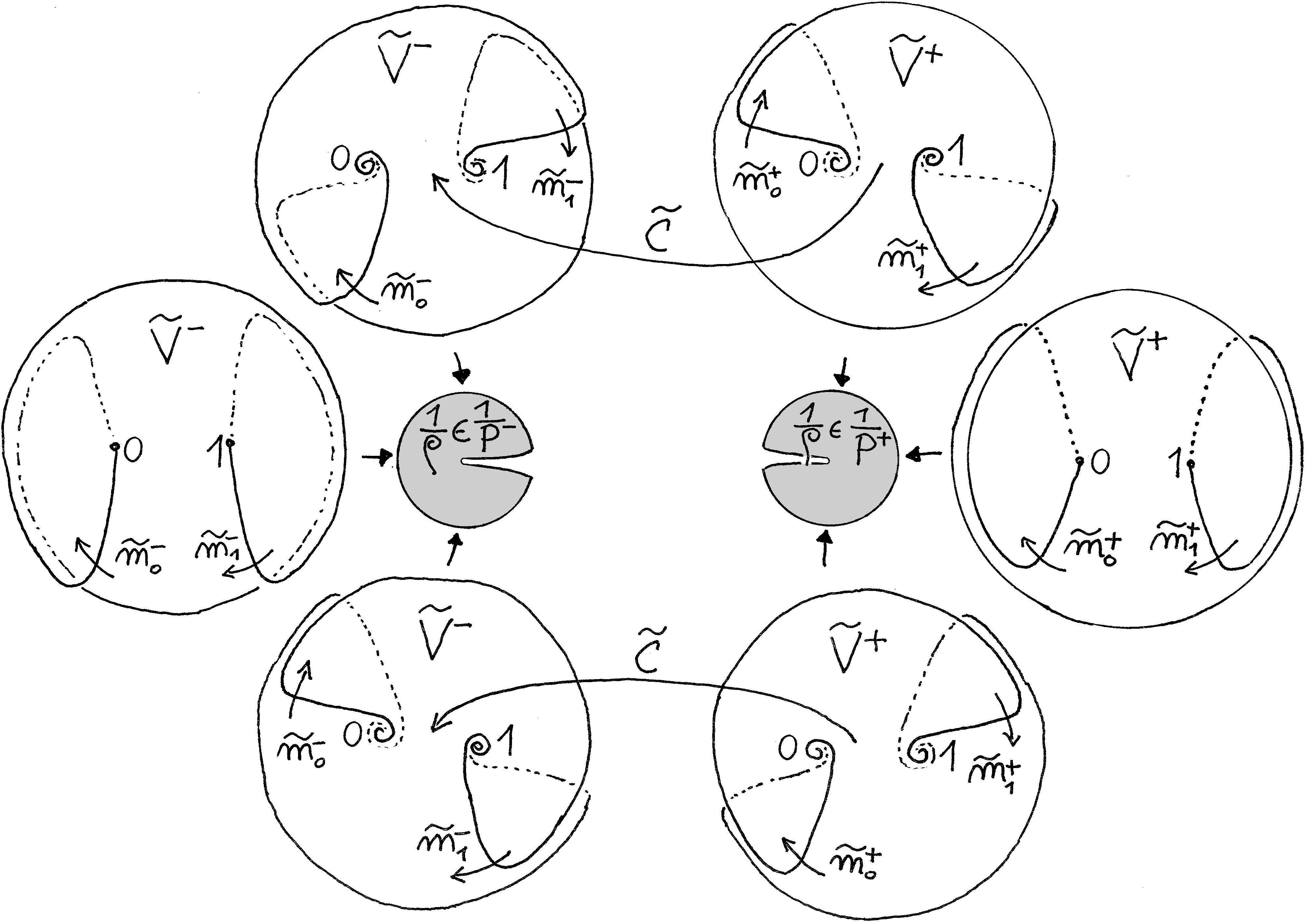}
\caption{The fundamental matrix solutions $\tilde V^\pm$ and their monodromy and transition matrices,
according to the values of $\frac{1}{\rho}$.
}
\label{figure:hg-V}
\end{figure}

\medskip
Monodromy matrices $\tilde m_0^+(\rho)$ and $\tilde m_1^+(\rho)$ of the solution $\tilde V^+$ around the singularities $0$ and $1$  in the positive direction from a base-point at $s=\frac{1}{2}$ are calculated in \cite{OTY}:
\begin{equation*}
\tilde m_0^+=
\begin{pmatrix}
e_1 &  & & \xi_1(e_1\!-\!1)\\[-3pt]
& \hskip-6pt\ddots &   & \vdots \\
& & \hskip-6pt e_{n-1} \hskip-3pt &  \xi_{n-1}(e_{n-1}\!-\!1)\hskip-3pt\\[3pt]
 &  &  & 1
\end{pmatrix},
\qquad
\tilde m_1^+=
\begin{pmatrix}
1 &   & &  \\[-3pt]
 & \ddots &   & \\ 
 & & \hskip-6pt 1 & \\[3pt]
\eta_1(e_n\!-\!1) \hskip-4pt & \ldots & \hskip-4pt \eta_{n-1}(e_n\!-\!1)\hskip-4pt & e_n
 \end{pmatrix},
\end{equation*}
where 
$\, e_j=e^{2\pi i(1-\beta_j+\rho)},\,\text{ for }\, j\leq n-1,\, \ e_n=e^{2\pi i(\gamma+\rho)},\,$ and
$$\xi_j=e^{\pi i(\gamma+\rho)} \frac{\Gamma(1\!+\!\gamma\!+\!\rho) \Gamma(\beta_j\!-\!\bmbeta{j})}
{\Gamma(\beta_j\!-\!\bm\alpha)},\qquad
\eta_j=e^{-\pi i(\gamma+\rho)} \frac{\Gamma(\!-\gamma\!-\!\rho) \Gamma(1\!-\!\beta_j\!+\!\bmbeta{j})}
{\Gamma(1\!-\!\beta_j\!+\!\bm\alpha)}.
$$
The connection matrix between the Floquet and the co-Floquet bases
$\tilde V^-(s,\rho)=\tilde V^+(s,\rho)\tilde C(\rho)$ is also calculated in \cite{OTY}:
\begin{equation*}
\tilde C=\begin{pmatrix}
1 &   & & \hskip-6pt -\xi_1 \\[-3pt]
 & \ddots &   & \vdots\\ 
 & & \hskip-6pt 1 & \hskip-6pt -\xi_{n-1}\hskip-3pt\\[3pt]
-\eta_1 \hskip-3pt & \ldots & \hskip-3pt -\eta_{n-1} \hskip-6pt & 1
 \end{pmatrix}.
\end{equation*}
Therefore the corresponding monodromy matrices of $\tilde V^-$ are equal to
$$\tilde m_\iota^-(\rho)=\tilde C(\rho)^{-1}\tilde m_\iota^+(\rho)\tilde C(\rho),\quad \iota=0,1,$$
\begin{equation*}
\tilde m_0^-=
\begin{pmatrix}
e_1 \hskip-12pt &  & & \\[-3pt]
& \hskip-6pt\ddots &   &  \\
& & \hskip-20pt e_{n-1}  & \\[3pt]
\eta_1(e_1\!-\!1) \hskip-4pt & \ldots & \hskip-4pt \eta_{n-1}(e_{n-1}\!-\!1)\hskip-4pt  & 1
\end{pmatrix},
\qquad
\tilde m_1^-=
\begin{pmatrix}
1 &  & &  \xi_1(e_n\!-\!1) \\[-3pt]
 & \ddots &   & \vdots\\ 
 & & 1 & \xi_{n-1}(e_n\!-\!1)\hskip-3pt\\[3pt]
& &  & e_n
 \end{pmatrix}.
\end{equation*}

\begin{remark}
 The Floquet and co-Floquet bases $\tilde V^\pm$ are both defined and analytic not only on $\rho\in P^\pm$ but for all $\rho\in\C$ except of a discrete set of resonant values accumulating at $\infty$.
\end{remark}
 
%The Floquet and co-Floquet bases $V^\pm(s,\rho)=\tilde V^\pm(s,\rho)R^{-1}$ are well-defined under a weaker assumption than \eqref{eq:hg-asump2}, that no two $\beta_j$'s differ by a non-zero integer.

\paragraph{The confluent family.}

Under the assumption \eqref{eq:hg-asump2} there are $n\!-\!1$ parameter-dependent singular solutions of the confluent equation \eqref{eq:hg-hgco} near $z=0$ 
%written as $$\tilde y_{\cdot j}^+(z,\rho)=z^{1-\beta_j}\,\tilde t_{\cdot j}^+(z,\rho), \qquad 1\leq j\leq n-1.$$ Their first component is 
are given for $\rho\neq\infty$ by %\,\footnote{Here $G_{*,*}^{*,*}$ is the Meijer $G$-function.}
\begin{align*}
\tilde y_{1j}^+(z,\rho)
&=z^{1-\beta_j}\,{}_nF_{n-1}\hg{\bm\alpha\!+\!1\!-\!\beta_j}{\bmbeta{j}\!+\!1\!-\!\beta_j}{\rho z}\\
&=z^{1-\beta_j}\, G_{n,n}^{n,1}\hg{1,\,\bmbeta{j}\!+\!1\!-\!\beta_j}{\bm\alpha\!+\!1\!-\!\beta_j}{\!-\frac{1}{\rho z}}\cdot \frac{\Gamma(\bmbeta{j}\!+\!1\!-\!\beta_j)}{\Gamma(\bm\alpha\!+\!1\!-\!\beta_j)},\qquad |\rho z|<1,\\
%&= e^{\pm\pi i(1-\beta_j)}\, G_{n,n}^{n,1}\hg{\beta_j,\,\bmbeta{j}}{\bm\alpha}{-\frac{1}{\rho z}}\cdot \frac{\Gamma(\bmbeta{j}\!+\!1\!-\!\beta_j)}{\Gamma(\bm\alpha\!+\!1\!-\!\beta_j)}, 
&\downarrow  \\
\tilde y_{1j}^+(z,\infty)&=z^{1-\beta_j}\,
G_{n-1,n}^{n,1}\hg{1,\,\bmbeta{n,j\!}\!+\!1\!-\!\beta_j}{\bm\alpha\!+\!1\!-\!\beta_j}{\!-\frac{1}{z}} \cdot \frac{\Gamma(\bmbeta{n,j\!}\!+\!1\!-\!\beta_j)}{\Gamma(\bm\alpha\!+\!1\!-\!\beta_j)},\quad |\arg z-\pi|<\frac{3\pi}{2},
%&=e^{\pm\pi i(1-\beta_j)}\, G_{n-1,n}^{n,1}\hg{\beta_j,\,\bmbeta{n,j}}{\bm\alpha}{-\frac{1}{z}} \cdot \frac{\Gamma(\bmbeta{n,j\!}\!+\!1\!-\!\beta_j)}{\Gamma(\bm\alpha\!+\!1\!-\!\beta_j)},
\end{align*}
with the limit asymptotic on compact sub-sectors to the divergent formal series
$$z^{1-\beta_j}\,{}_nF_{n-2}\hg{\bm\alpha\!+\!1\!-\!\beta_j}{\bmbeta{n,j\!}\!+\!1\!-\!\beta_j}{z}
=z^{1-\beta_j}\,\sum_{k=0}^{+\infty}\frac{(\bm\alpha\!+\!1\!-\!\beta_j)_k}{(\bmbeta{n,j}\!+\!1\!-\!\beta_j)_k(1)_k}z^k.$$
And the singular solution at $z=\frac{1}{\rho}$ is 
%$$\tilde y_{\cdot n}^+(z,\rho)= z^{-\rho}(z\!-\!\tfrac{1}{\rho})^{\gamma+\rho+n-1}\, \tilde t_{\cdot n}^+(z,\rho),$$ where the first component 
is given by
\begin{align*}
\tilde y_{1n}^+(z,\rho)
%&=z^{-\rho}(z\!-\!\tfrac{1}{\rho})^{\gamma+\rho}\,\sum_{k=0}^{+\infty}\frac{(-1)^k(\rho)^{k+n-1} c_k}{(\gamma+\rho+1)_{k+n-1}}(z-\tfrac{1}{\rho})^{k+n-1}\\
&= G_{n,n}^{n,0}\hg{\bm\beta}{\bm\alpha}{\frac{1}{\rho z}} \cdot\rho^{-\gamma} \,\Gamma(1\!+\!\gamma\!+\!\rho),\qquad |\rho z|>1\\
&\downarrow\\
\tilde y_{1n}^+(z,\infty)&= G_{n-1,n}^{n,0}\hg{\bmbeta{n}}{\bm\alpha}{\frac{1}{z}}, \qquad |\arg z|<\frac{3\pi}{2},
\end{align*}
which is asymptotic on compact sub-sectors to the divergent formal series
\begin{equation*}
e^{-\frac{1}{z}}z^{\gamma}\sum_{k=0}^{+\infty}(-1)^k c_k z^{k+n-1},
\qquad\text{with $c_k$ as in \eqref{eq:hg-ck}.}
\end{equation*}

The constructed fundamental matrix solution 
$\tilde Y^+=(\tilde y_{ij}^+)$ have the asymptotic behavior
$$\tilde Y^+(z,\rho)\sim \left(R+O(1)(z-\tfrac{B}{\rho})\right)(z-\tfrac{B}{\rho})^{\tilde A_D+\rho}z^{-\rho}.$$
Therefore
$$Y^{+\!}(z,\rho):=\tilde Y^{+\!}(z,\rho)R^{-1}\sim \left(I+O(1)(z-\tfrac{B}{\rho})\right)(z-\tfrac{B}{\rho})^{A_D+\rho}z^{-\rho}$$ 
is the corresponding Floquet bases of the confluent family \eqref{eq:hg-AB} with the right asymptotic behavior. 

In the definition of $\tilde Y^+$ above the right choice of branch of $\log z$
and $\log(z-\frac{1}{\rho})$ in $\tilde\Phi(z,\rho)=(z-\tfrac{B}{\rho})^{\tilde A_D+\rho}z^{-\rho}$ is of essential importance.
Let $\alpha$ be a direction, $-\pi<\alpha<0$, and chose the bases  $\tilde Y^\pm_{[\alpha]}$ and $\tilde Y^\pm_{[\alpha+\pi]}$ so that they are related to each other as in Figure~\ref{figure:hg-Y},
$$\tilde Y_{[\alpha+\pi]}^+(z,\rho)=(\rho z)^{-\rho}\tilde V^+(\rho z, \rho)\rho^{-\tilde A_D},\quad
\tilde Y_{[\alpha]}^-(z,\rho)=(\rho z)^{-\rho}\tilde V^-(\rho z, \rho)\rho^{-\tilde A_D}.$$

\smallskip

\begin{figure}[t]
\centering
\includegraphics[width=\textwidth]{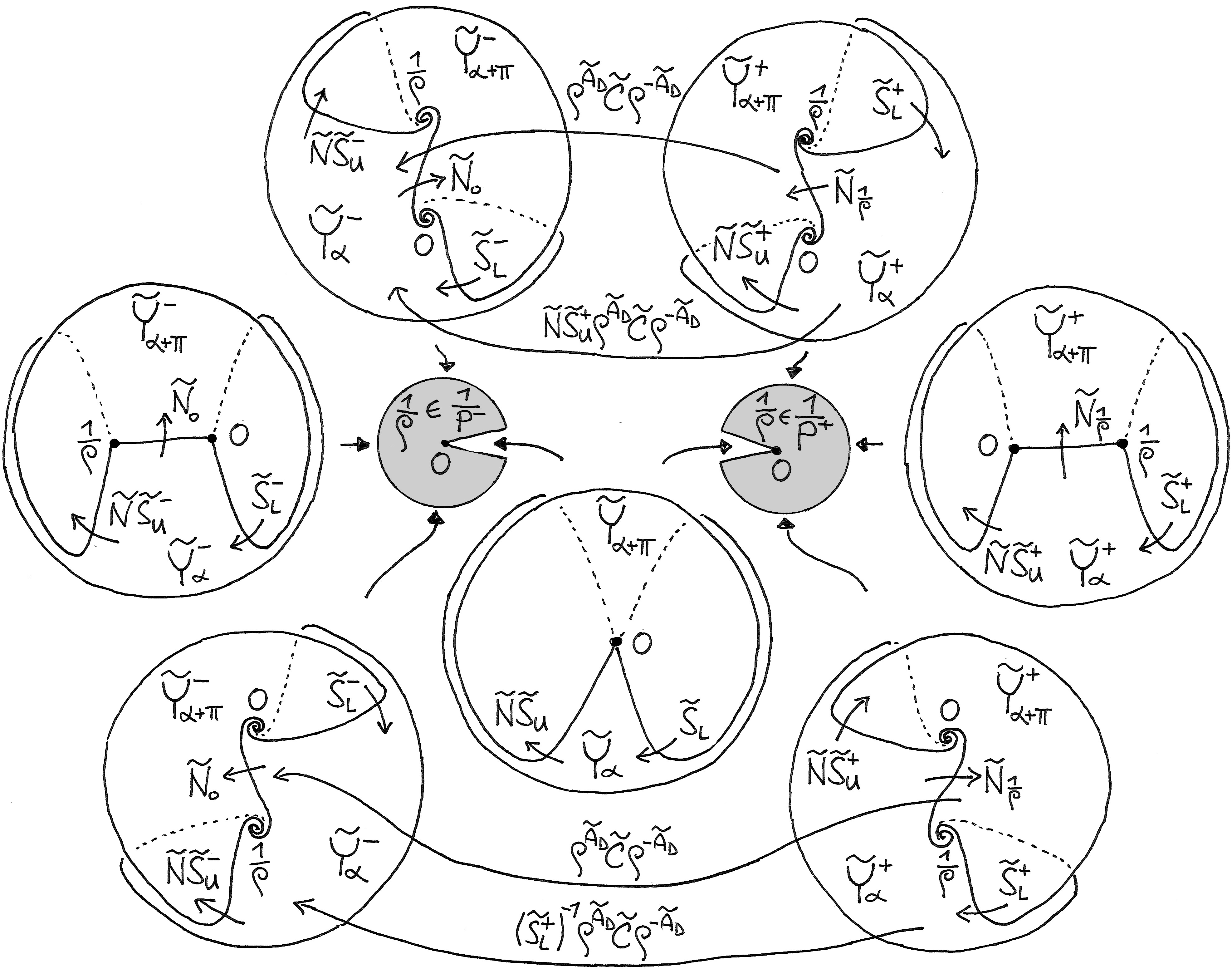}
\caption{The fundamental matrix solutions $\tilde Y^\pm_{[\alpha]}$,
$\tilde Y^\pm_{[\alpha+\pi]}$ on their natural domains (restricted to a fixed neighborhood of 0) and their transition matrices,
according to the values of $\frac{1}{\rho}$.
The limit system is in the center. (See \cite{LR2, HLR, Kl2} for more details on the construction of these ramified domains.)}
\label{figure:hg-Y}
\end{figure}

Then the monodromy matrices of the fundamental matrix solution 
$\tilde Y_{[\alpha+\pi]}^+(z,\rho)$,
resp. $\tilde Y_{[\alpha]}^-(z,\rho)$,  around 
$0$ and $\frac{1}{\rho}$  ($\rho\neq\infty$)  in the positive direction from a base-point at $z=\frac{1}{2\rho}$ 
 are equal 
\begin{equation}\label{eq:hg-M+}
\tilde M_0^+=\tilde N_0\tilde S_U^+
=e^{-2\pi i\rho}\rho^{\tilde A_D}\tilde m^+_0\rho^{-\tilde A_D},\qquad
\tilde M_{\frac{1}{\rho}}^+=\tilde S_L^+\tilde N_{\frac{1}{\rho}} 
=\rho^{\tilde A_D}\tilde m^+_1\rho^{-\tilde A_D},
\end{equation}
resp.
\begin{equation}\label{eq:hg-M-}
\tilde M_0^-=\tilde N_0\tilde S_L^-
=e^{-2\pi i\rho}\rho^{\tilde A_D}\tilde m^-_0\rho^{-\tilde A_D},\qquad
\tilde M_{\frac{1}{\rho}}^-=\tilde N\tilde S_U^-\tilde N_0^{-1} 
=\rho^{\tilde A_D}\tilde m^-_1\rho^{-\tilde A_D},
\end{equation} 
where 
\begin{equation*}
\tilde N_0=
\begin{pmatrix}
e^{2\pi i(1-\beta_1)} \hskip-18pt  &  & & \\[-3pt]
& \hskip-6pt\ddots &   & \\
 & &  e^{2\pi i(1-\beta_{n-1})} \hskip-24pt &  \\
 &  &  & e^{-2\pi i\rho}
\end{pmatrix},
\qquad
\tilde N_{\frac{1}{\rho}}=
\begin{pmatrix}
1  &   & & \\[-3pt] 
& \ddots &   & \\
 & & 1 &  \\
 &  &  &  e^{2\pi i(\gamma+\rho)}
\end{pmatrix},
\end{equation*}
and $\,\tilde N=\tilde N_0\tilde N_{\frac{1}{\rho}},\,$
are monodromies of the fundamental matrix solution $\tilde\Phi(z,\rho)=z^{-\rho}(z-\frac{B}{\rho})^{\tilde A_D+\rho}$ of the diagonal model system, and 
\begin{equation*}
\tilde S_U^\pm:=\tilde S_{\alpha+2\pi,\alpha+\pi}^\pm=
\begin{pmatrix}
1 &    & & \tilde s^\pm_{1n}\ \\[-3pt]
 & \hskip-3pt\ddots\hskip-3pt &   & \hskip3pt\vdots\ \\ 
 & & 1 \hskip-6pt & \tilde s^\pm_{n-1,n}\hskip-3pt \\[3pt]
  &  &  & 1 
\end{pmatrix},
\qquad
\tilde S_L^\pm:=\tilde S_{\alpha+\pi,\alpha}^\pm=
\begin{pmatrix}
1 &   & &  \\[-3pt]
 & \ddots &   & \\ 
 & & \hskip-6pt 1 & \\[3pt]
\tilde s^\pm_{n1} \hskip-6pt & \ldots & \hskip-6pt\tilde s^\pm_{n,n-1}\hskip-6pt  & 1
 \end{pmatrix}
\end{equation*}
are the Stokes matrices.
It follows from \eqref{eq:hg-M+} that the Stokes multipliers $\tilde s^+_{ij}(\rho)$ are equal to:
\begin{align*}
\tilde s^+_{jn}(\rho) 
&= \xi_j(\rho)\,(1\!-\!e^{-2\pi i(\rho+1-\beta_j)})\,\rho^{1-\beta_j-\gamma}\\
&=-2\pi i\, e^{\pi i(\gamma+\beta_j+n)} \frac{\Gamma(\beta_j\!-\!\bmbeta{n,j})}{\Gamma(\beta_j\!-\!\bm\alpha)} 
\cdot \rho^{1-\beta_j-\gamma} \frac{\Gamma(1\!+\!\gamma\!+\!\rho)}{\Gamma(2\!-\!\beta_j\!+\!\rho)},\\
&\downarrow\\
\tilde s^+_{jn}(\infty)&=-2\pi i\, e^{\pi i(\gamma+\beta_j+n)} \frac{\Gamma(\beta_j\!-\!\bmbeta{n,j})}{\Gamma(\beta_j\!-\!\bm\alpha)},  
\end{align*}
$\big(\,\text{since} \ \lim_{P^+\ni\rho\to\infty} \rho^{\gamma+\beta_j-1}\frac{\Gamma(2-\beta_j+\rho)}{\Gamma(1+\gamma+\rho)}=1 \,\big),$
and
\begin{align*}
\tilde s^+_{nj}(\rho) 
&=\eta_j(\rho)\,(e^{2\pi i(\gamma+\rho)}\!-\!1)\,\rho^{\gamma+\beta_j-1}\\
&= -2\pi i \frac{\Gamma(1\!-\!\beta_j\!+\!\bmbeta{n,j})}{\Gamma(1\!-\!\beta_j\!+\!\bm\alpha)}
\cdot \rho^{\gamma+\beta_j-1}\frac{\Gamma(2\!-\!\beta_j\!+\!\rho)}{\Gamma(1\!+\!\gamma\!+\!\rho)},\\
&\downarrow\\
\tilde s^+_{nj}(\infty)&=-2\pi i \frac{\Gamma(1\!-\!\beta_j\!+\!\bmbeta{n,j})}{\Gamma(1\!-\!\beta_j\!+\!\bm\alpha)}.
\end{align*}
From \eqref{eq:hg-M-} one then obtains the Stokes multipliers $\tilde s^-_{ij}(\rho)$:
\begin{align*}
\tilde s^-_{jn}(\rho) 
&=\xi_j(\rho)\, e^{-2\pi i(1-\beta_j+\rho)}(e^{2\pi i(\gamma+\rho)}\!-\!1)\,\rho^{1-\beta_j-\gamma}\\
&=-2\pi i\, e^{\pi i(\gamma+\beta_j+n)} \frac{\Gamma(\beta_j\!-\!\bmbeta{n,j})}{\Gamma(\beta_j\!-\!\bm\alpha)} 
\cdot \left(e^{- \pi i}\rho\right)^{1-\beta_j-\gamma} \frac{\Gamma(\beta_j\!-\!1\!-\!\rho)}{\Gamma(\!-\gamma\!-\!\rho)},\\
&\downarrow\\
\tilde s^-_{jn}(\infty)&=-2\pi i\, e^{\pi i(\gamma+\beta_j+n)} \frac{\Gamma(\beta_j\!-\!\bmbeta{n,j})}{\Gamma(\beta_j\!-\!\bm\alpha)},  
\end{align*}
$\big(\,\text{since}\ \lim_{P^-\ni\rho\to\infty} (e^{- \pi i}\rho)^{\gamma+\beta_j-1} \frac{\Gamma(-\gamma-\rho)}{\Gamma(\beta_j-1-\rho)}=1\,\big)$,
and
\begin{align*}
\tilde s^-_{nj}(\rho) 
&=\eta_j(\rho)\,(e^{2\pi i(1-\beta_j+\rho)}\!-\!1)\,\rho^{\gamma+\beta_j-1}\\
&= -2\pi i \frac{\Gamma(1\!-\!\beta_j\!+\!\bmbeta{n,j})}{\Gamma(1\!-\!\beta_j\!+\!\bm\alpha)}
\cdot \left(e^{- \pi i}\rho\right)^{\gamma+\beta_j-1} \frac{\Gamma(\!-\gamma\!-\!\rho)}{\Gamma(\beta_j\!-\!1\!-\!\rho)},\\
&\downarrow\\
\tilde s^-_{nj}(\infty)&=-2\pi i \frac{\Gamma(1\!-\!\beta_j\!+\!\bmbeta{n,j})}{\Gamma(1\!-\!\beta_j\!+\!\bm\alpha)}.
\end{align*}

One could also proceed the opposite way:
The Stokes matrices $\tilde S_\bullet(\infty)$, $\bullet=U,L$, of the limit generalized confluent hypergeometric equation have been calculated in \cite{DM, KO}, 
and the Stokes matrices  $\tilde S_\bullet^\pm(\rho)$ of the family are related to them via Proposition~\ref{prop:hg-stokesmatrices}.

Note that the confluent Floquet and co-Floquet bases $Y^\pm(z,\rho)=\tilde Y^\pm(z,\rho)R^{-1}$ and their monodromies, resp. Stokes matrices
$M_i^\pm(\rho)=R\tilde M_i^\pm(\rho)R^{-1}$,  $\iota=0,\frac{1}{\rho}$,
resp. $S_\bullet^\pm(\rho)=R\tilde S_\bullet^\pm(\rho)R^{-1}$,
$\bullet=U,L$,
are well-defined under a weaker assumption than \eqref{eq:hg-asump2}, that no two $\beta_j$'s differ by a non-zero integer.

\goodbreak


\small
\begin{thebibliography}{AAA}


\bibitem[Ba]{Ba} W. Balser, \textit{Formal power series and linear systems of meromorphic ordinary differential equations}, Springer, 2000. 		

\bibitem[BJL]{BJL} W. Balser, W.B. Jurkat, D.A. Lutz, \textit{On the reduction of connection problems with an irregular singularity to ones with only regular singularities, I, II}, SIAM J. Math. Anal. \textbf{12} (1981), 691--721,
SIAM J. Math. Anal. \textbf{19} (1988), 398--443.


%\bibitem[BJL]{BJL} W. Balser, W.B. Jurkat, D.A. Lutz, \textit{Birkhoff Invariants and Stokes' Multipliers for Meromorphic Linear Differential Equations},  J. Math. Anal. Appl. \textbf{71} (1979), 48--94.

%\bibitem[BD]{BD} B. Branner, K. Dias, \textit{Classification of complex polynomial vector fields in one complex variable}, J. Diff. Eq. Appl., \textbf{16} (2010), 463--517. 

%\bibitem[BV]{BaVa}  D.G. Babbitt, V.S. Varadarajan, \textit{Local moduli for meromorphic differential equations}, Astérisque \textbf{169-170} (1989).

\bibitem[Du]{Du} A. Duval, \textit{Confluence procedures in the generalized hypergeometric family}, J. Math. Sci. Univ. Tokyo \textbf{5} (1998), 597--625.

%\bibitem[Du]{Du} A. Duval, \textit{Biconfluence et groupe de Galois}, J. Fac. Sci. Univ. Tokyo, Sect. IA, Math. \textbf{38} (1991), 211--223.

\bibitem[DM]{DM} A. Duval, C. Mitschi, \textit{Matrices de Stokes et groupe de Galois des équations hypergéométriques généralisées}, Pacific J. of Math. \textbf{138} (1989), 25--56.


%\bibitem[DES]{DES} A. Douady, F. Estrada, P. Sentenac, \textit{Champs de vecteurs polynômiaux sur $\C$}, unpublished manuscript, (2005).

%\bibitem[Gan]{Gan} F.R. Gantmacher, \textit{The theory of matrices}, New York, 1959.

\bibitem[Gar]{Gar} R. Garnier, \textit{Sur les singularités irrégulières des équations différentielles linéaires}, J. de  math. pures et appl. $8^e$ série (1919), 99--200.

\bibitem[Gl1]{Gl1} A. Glutsyuk, \textit{Stokes Operators via Limit Monodromy of Generic Perturbation}, J. Dyn. Control Syst. \textbf{5} (1999), 101--135.

\bibitem[Gl2]{Gl2}  A. Glutsyuk, \textit{Resonant Confluence of Singular Points and Stokes Phenomena}, J. Dyn. Control Syst. \textbf{10} (2004), 253--302.	

\bibitem[Hu]{Hu} M. Hukuhara, \textit{Développements Asymptotiques des Solutions Principales d'un Systeme Différentiel Linéaire du Type Hypergéométrique}, Tokyo J. of Math. \textbf{05} (1982), 491--499.

\bibitem[HLR]{HLR} J. Hurtubise, C. Lambert, C. Rousseau, \textit{Complete system of analytic invariants for unfolded differential linear systems with an irregular singularity of Poincaré rank k}, Moscow Math. J. \textbf{14} (2013), 309--338. 

%\bibitem[IY]{IlYa} Iwasaki K., Kimura H., Shimomura S., Yoshida M., \textit{From Gauss to Painlevé: a modern theory of special functions}, Vieweg, 1991.	

\bibitem[IY]{IY} Y. Ilyashenko, S. Yakovenko, \textit{Lectures on Analytic Differential Equations}, Grad. Studies Math. \textbf{86}, Amer. Math. Soc., Providence, 2008.	

%\bibitem[JLP1]{JLP1}  W.B. Jurkat, D.A. Lutz, A. Peyerimhoff, \textit{Birkhoff invariants and effective calculations for meromorphic linear differential equations I}, J. Math. Anal. Appl. \textbf{53} (1976), 	438--470.

%\bibitem[JLP2]{JLP2} W.B. Jurkat, D.A. Lutz, A. Peyerimhoff, \textit{Birkhoff invariants and effective calculations for meromorphic linear differential equations II}, Houston J. Math. \textbf{2} (1976), 			207--238.

%\bibitem[KT]{KT} H. Kimura, K. Takano, \textit{On confluences of general hypergeometric systems}, Tohoku Math. J.  \textbf{58} (2006), 1--31.

\bibitem[Kl1]{Kl1} M. Klime\v{s}, \textit{Analytic classification of families of linear differential systems unfolding a resonant irregular singularity}, preprint arXiv:1301.5228.

\bibitem[Kl2]{Kl2} M. Klime\v{s}, \textit{Confluence of singularities of non-linear differential equations via Borel-Laplace transformations}, to be published in J. Dynam. Contr. Syst. [arXiv:1307.8383]

\bibitem[Ko1]{Ko1} M. Kohno, \textit{Frobenius' theorem and Gauss-Kummer's formula}, Funkcialaj Ekvacioj \textbf{28} (1985), 249--266.

\bibitem[Ko2]{Ko2} M. Kohno, \textit{Global analysis in linear differential equations}, Math. and its appl. \textbf{471}, Kluwer Acad. Publ., Dodrecht, 1999.

\bibitem[KO]{KO} M. Kohno, S. Ohkohchi, \textit{Generalized hypergeometric equations of non-Fuchsian type}, Hiroshima Math. J. \textbf{13} (1983), 83--100.

%\bibitem[Ko]{Ko} V.P. Kostov, \textit{Normal forms of unfoldings of non-fuchsian systems}, C. R. Acad. Sc. Paris \textbf{318} (1994), 623--628.

\bibitem[Le]{Le} A.H.M. Levelt, \textit{Hypergeometric functions}, dissertation, Nederl. Akad. Wetensch. Proc. Ser. A \textbf{64}, 1961.

\bibitem[LR1]{LR1}  C. Lambert, C. Rousseau, \textit{The Stokes phenomenon in the confluence of the hypergeometric equation using Riccati equation}, J. Differential Equations \textbf{244} (2008) 2641-–2664.

\bibitem[LR2]{LR2}  C. Lambert, C. Rousseau, \textit{Complete system of analytic invariants for unfolded differential linear systems with an irregular singularity of Poincaré rank 1}, Moscow Math. Journal \textbf{12} (2012), 77--138.

\bibitem[Lu]{Lu} Y.L. Luke, \textit{Special functions and their approximations I, II}, Math. Sci. Eng. \textbf{53}, Acad. Press, New York, 1969.

%\bibitem[Ma]{Ma} B. Malgrange, \textit{Sommation des séries divergentes}, Expositiones Mathematicae \textbf{13} (1995), 163--222.

\bibitem[MR]{MR} J. Martinet, J.-P. Ramis, \textit{Théorie de Galois differentielle et resommation}, in: \textit{Computer Algebra and Differential Equations} (E.Tournier ed.), Acad. Press (1988).

\bibitem[Ok1]{Ok1} K. Okubo, \textit{An Extension of Gauss' Formula for Hypergeometric Series}, \begin{CJK}{UTF8}{min} 数理解析研究所講究録 \end{CJK} \textbf{105} (1971), 53--57.
[http://hdl.handle.net/2433/106318]

\bibitem[Ok2]{Ok2} K. Okubo, \textit{Connection problems for systems of linear differential equations}, Lect. Notes in Math. \textbf{243}, Springer, 1971.

\bibitem[OTY]{OTY} K. Okubo, K. Takano, S. Yoshida, \textit{A connection problem for the generalized hypergeometric equation}, Funkcialaj Ekvacioj \textbf{31} (1988), 483--495.

\bibitem[Pa]{Pa}  L. Parise, \textit{Confluence de singularités régulières d'équations différentielles en une singularité irrégulière. Modèle de Garnier}, thèse de doctorat, IRMA Strasbourg (2001). [http://www-irma.u-strasbg.fr/annexes/publications/pdf/01020.pdf]

\bibitem[Ra]{Ra} J.-P. Ramis, \textit{Confluence et résurgence}, J. Fac. Sci. Univ. Tokyo, Sec. IA Math. \textbf{36} (1989), 703--716. 

\bibitem[Sch1]{Sch1}  R. Sch\"afke, \textit{\"Uber das globale analytische Verhalten der Normall\"osungen von $(s-B)v'(s)=(B+t^{-1}A)v(s)$ und zweier Arten von assoziierten Funktionen}, Math. Nachr. \textbf{121} (1985), 123--145.

\bibitem[Sch2]{Sch2}  R. Sch\"afke, \textit{Confluence of several regular singular points into an irregular singular one}, J. Dyn. Control Syst. \textbf{4} (1998), 401--424.

%\bibitem[Si]{Sib} Y. Sibuya, \textit{Linear differential equations in the complex domain : problems of analytic continuation}, Translations of Mathematical Monographs \textbf{82}, American Mathematical 			Society, Providence, 1990.

%\bibitem[Wa]{Wa} W. Wasow, \textit{Asymptotic Expansions for Ordinary Differential Equations}, John Wiley and Sons Inc., 1966.

\bibitem[Zh]{Zh} C. Zhang, \textit{Confluence et phenomène de Stokes}, J. Math. Sci. Univ. Tokyo \textbf{3} (1996), 91--107.

\end{thebibliography}
\end{document}